\documentclass[12pt,a4paper]{article}
\usepackage[english,frenchb]{babel}
\usepackage[latin1]{inputenc}
\usepackage{amsmath,amscd}
\usepackage{amsfonts}
\usepackage{amssymb}
\usepackage{amsthm}

\usepackage{hyperref}

\usepackage[all]{xy}
\usepackage[usenames,dvipsnames]{color}
\usepackage[dvips]{graphicx}

\newtheorem{maintheorem}{Theorem}

\newtheorem{teo}{Theorem}[section]
\newtheorem{df}[teo]{Definition}

\newtheorem{lem}[teo]{Lemma}

\newtheorem{rem}[teo]{Remark}

\newtheorem{question}[teo]{Question}

\newcommand{\D}{\mathbb D}
\newcommand{\Hyp}{\mathbb H}
\newcommand{\R}{\mathbb R}
\newcommand{\DD}{\mathcal D}
\newcommand{\FF}{\mathcal F}
\newcommand{\PP}{\mathcal P}
\newcommand{\RR}{\mathcal R}
\newcommand{\TT}{\mathcal T}
\newcommand{\CPun}{\mathbb{C}\mathbb{P}^1}

\newcommand{\PSLC}{PSL_2(\mathbb{C})}
\newcommand{\PSLR}{PSL_2(\mathbb{R})}
\newcommand{\tS}{\widetilde{S}}
\newcommand{\tphi}{\widetilde{\phi}}
\newcommand{\tpsi}{\widetilde{\psi}}
\newcommand{\bigslant}[2]{{\raisebox{.2em}{$#1$}\left/\raisebox{-.2em}{$#2$}\right.}}

\long\def\symbolfootnote[#1]#2{\begingroup
\def\thefootnote{\fnsymbol{footnote}}\footnote[#1]{#2}\endgroup} 

\title{Heijal's theorem for parabolic projective structures on non compact surfaces.}
\author{Nicolas Hussenot Desenonges}

\begin{document}
\maketitle
\selectlanguage{english}

\begin{abstract}
We prove that the holonomy map from the set of equivalence classes of parabolic complex projective structures on non compact surfaces to the set of equivalence classes of parabolic representations of the fundamental group of the surface to $\PSLC$ is a local biholomorphism.
\end{abstract}

\section{Introduction.}

\paragraph{Complex projective structure, monodromy map.}
Let $S$ be an oriented surface. A \textit{complex projective structure} on $S$ is a $(\PSLC,\CPun)$-structure on $S$, i.e. the data of a maximal atlas of charts mapping open sets of $S$ into $\CPun$ such that the transition functions are restrictions of M\"obius transformations. Two complex projective structures $\sigma_1$ and $\sigma_2$ are equivalent if there exists an orientation-preserving diffeomorphism of $S$ homotopic to the identity that pulls back the projective charts of $\sigma_2$ to projective charts of $\sigma_1$. We denote by $\PP(S)$ the classes of complex projective structures for the previous equivalence relation.

As M\"obius transformations are holomorphic, a complex projective structure on $S$ determines a complex structure on $S$. Hence, there is a natural map, sometimes called \textit{forgetful map}:
$$\pi:\PP(S)\longrightarrow \TT(S),$$
where $\TT(S)$ is the Teichm\"uller space of $S$.

Given a complex projective structure on $S$, we can define a pair $(\DD,\rho)$ where $\rho:\pi_1(S)\to\PSLC$ is a morphism called \textit{holonomy representation} which globalizes the transition functions and $\DD:\tS\to\CPun$ is a $\rho$-equivariant local homeomorphism\footnote{$\DD$ is holomorphic for the underlying complex structure on $S$} called \textit{developing map} which globalizes the projective charts. The map $\DD$ is well defined up to post-composition by a M\"obius transformation and the representation $\rho$ is well defined up to conjugacy by this M\"obius transformation. Hence, we have a natural map, called \textit{holonomy map}:
$$Hol:\PP(S)\longrightarrow\RR(S)$$
where $\RR(S)$ is the set of equivalence classes of morphisms $\rho:\pi_1(S)\to\PSLC$ modulo the action by conjugacy.

\paragraph{The compact case.} When $S$ is the compact oriented surface $S_g$ with genus $g\geq 2$, there are many results about the behaviour of the map $Hol$:

\begin{teo}[Poincar\'e]\label{Poincare}
In restriction to a fiber of the forgetful map $\pi$, the map $Hol$ is injective.
\end{teo}

The spaces $\PP(S)$ and $\RR(S)$ can be given the structure of a complex manifold (see preliminaries). Heijal \cite{He} proved that $Hol$ is a local homeomorphism. It was proved by Earle \cite{Ea} and Hubbard \cite{Hu} that $Hol$ is holomorphic. Hence, we have:

\begin{teo}[Heijal, Earle, Hubbard]\label{Heijal}
The map $Hol$ is a local biholomorphism.
\end{teo}

The holonomy map is not surjective. But the celebrated theorem of Gallo-Kapovic-Marden \cite{GKM} characterizes exactly the image of the holonomy map:

\begin{teo}[Gallo-Kapovic-Marden]\label{GKM}
 $Hol(\mathcal{P}(S_g))$ is exactly the classes of representations $\rho$ satisfying:
\begin{itemize}
\item $\rho$ is non elementary (i.e. there is no Borel probability measure on $CP^1$ invariant by the action of the holonomy group $\rho(\pi_1(S))$).
\item $\rho$ can be lifted to a representation $\tilde{\rho}:\pi_1(S)\rightarrow SL(2,\mathbb{C})$.
\end{itemize}
\end{teo}

The holonomy map is not globally injective: given a projective structure, there is a surgery called $(2\pi-)$grafting, which produces a new projective structure without changing the holonomy representation. The following theorems of Goldman \cite{Go} and Baba \cite{Ba} say more or less that grafting is the only way to do this:

\begin{teo}\label{Goldman}
\begin{itemize}
\item Goldman: Every complex projective structure with quasi-Fuschian holonomy representation can be obtained by grafting the uniformizing projective structure. 
\item Baba: Given two projective structures with the same purely loxodromic holonomy representation, we can pass from one to the other by grafting and ungrafting.
\end{itemize}
\end{teo}

\paragraph{The parabolic case.} In this paper, we study the following natural question: if instead of considering the compact surface with genus $g\geq 2$, we consider the compact surface with genus $g$ and $n$ points deleted (with $3g-3+n>0$), are the theorems of the previous part still valid? In this general context, the answer is no! For example, the space of complex projective structure $\PP(S_{g,n})$ on $S_{g,n}$ is infinite dimensional whereas the space of representations of $\pi_1(S_{g,n})$ in $\PSLC$ modulo conjugacy is finite dimensional. We deduce immediatly that the analogous of Heijal's and Poincar\'e's theorems cannot be true. So we need to restrict our space $\PP(S_{g,n})$ in order to hope something. It is the goal of the following definition:
\begin{df}\label{defpara}
A complex projective structure on $S_{g,n}$ is said to be parabolic if for any puncture $p$, there is a local coordinate $z$ around $p$ and a M\"obius transformation $A $ such that $A\circ \mathcal{D}(z)=\frac{1}{2i\pi}\log z$.
\end{df}
In the previous definition, the developing map has to be thought as a multivalued map $\DD:S_{g,n}\rightarrow\CPun$. The set of parabolic complex projective structures modulo the equivalence relation defined at the beginning of the introduction, is denoted by $\PP_{para}(S_{g,n})$. 

Note that for a parabolic complex projective structure, in the local coordinate $z$ around a puncture $p$ given by definition \ref{defpara}, we have $A\circ \mathcal{D}(e^{2i\pi}\cdot z)=A\circ\DD(z)+1$. Hence $\mathcal{D}(e^{2i\pi}\cdot z)=A^{-1}\circ T_1\circ A\circ\DD(z)$ where $T_1:z\mapsto z+1$ is the translation. This means that the image of an element of $\pi_1(S_{g,n})$ which makes one loop around a puncture is a parabolic element of $\PSLC$ (i.e. conjugated to $z\mapsto z+1$). Such a representation is called parabolic and the quotient of all such representations by the conjugacy action is denoted by $\RR_{para}(S_{g,n})$. Hence we have a map: 
$$Hol:\PP_{para}(S_{g,n})\longrightarrow\RR_{para}(S_{g,n}).$$

The analogous of Poincar\'e's theorem has been proved by I. Kra \cite{Kr}, namely:

\begin{teo}[Kra]
In restriction to a fiber of the forgetful map $\pi$, the map $Hol$ is injective.
\end{teo}

As in the compact case, the two spaces $\PP_{para}(S_{g,n})$ and $\RR_{para}(S_{g,n})$ can be endowed with a structure of complex manifolds (see section \ref{paradim}). The goal of this paper is to prove that the analogous of Heijal's theorem also holds:

\begin{maintheorem}\label{theoprincipal}
The holonomy map $Hol$ is a local biholomorphism.
\end{maintheorem}

\paragraph{Questions.}

\begin{question}
Is the analogous of Gallo-Kapovic-Marden theorem \ref{GKM} true in this context? The fundamental group of $S_{g,n}$ (with $n>0$) being a free group, any morphism $\rho:\pi_1(S_{g,n})\to \PSLC$ can be lifted to a morphism $\tilde{\rho}:\pi_1(S_{g,n})\rightarrow SL(2,\mathbb{C})$. Hence the precise question is: given a non-elementary parabolic representation $\rho$, does there exist a parabolic complex projective structure with holonomy $\rho$? See also \cite[Problem 12.2.1]{GKM} for a question in the same vein (but slightly different).
\end{question}

\begin{question}
Are the analogous of Goldman's and Baba's theorems \ref{Goldman} true in our context?
\end{question}



\section{Preliminaries.}\label{Prel}

\paragraph {Complex projective structures.} Let $S$ be an oriented surface. A \textit{complex projective structure} on $S$ is the data of a pair $(\DD,\rho)$ where $\rho:\pi_1(S)\to\PSLC$ is a morphism and $\DD:\tS\to\CPun$ is a $\rho$-equivariant local homeomorphism. Two complex projective structures $(\DD_1,\rho_1)$ and $(\DD_2,\rho_2)$ are said to be equivalent if there exists a $\pi_1(S)$-equivariant homeomorphism $\widetilde{\phi}$ of $\widetilde{S}$\footnote{$\widetilde{\phi}\circ\gamma=\gamma\circ\widetilde{\phi}$, $\forall\gamma\in\pi_1(S)$} and a M\"obius transformation $A$ such that:
\begin{itemize}
\item $\DD_2=A\circ\DD_1\circ\tphi$
\item $\rho_2=A\circ\rho_1\circ A^{-1}$
\end{itemize}
The set of equivalence classes of complex projective structures on $S$ is denoted by $\PP(S)$. It is a classical fact that this definition is equivalent to the one given in the introduction.

\paragraph {Parabolic type projective structure.} Let $S_{g,n}$ be the topological surface obtained from the oriented compact surface of genus $g$ by deleting $n$ points. Fix a universal covering map: $proj:\widetilde{S_{g,n}}\longrightarrow S_{g,n}$. Let $(\DD,\rho)$ be a complex projective structure on $S_{g,n}$. Thinking $\DD$ as a multivalued map, $(\DD,\rho)$ is said to be \emph{parabolic} at a puncture $p$ if there is a neighborhood $V$ of $p$, a homeomorphism $\phi:V\to \mathbb{D}^*$ and a M\"obius transformation $A$ such that:
$$A\circ\DD\circ\phi^{-1}(z)=\frac{1}{2i\pi}\log z.$$
Equivalently, $(\DD,\rho)$ is parabolic at $p$ if there exists:
\begin{itemize}
\item a neighborhood $V=\bigslant{C}{<\gamma>}$ of $p$, where $C$ is a connected component of $proj^{-1}(V)$ and $\gamma\in\pi_1(S_{g,n})$,
\item a homeomorphism $\tphi:C\to \Hyp$ with $\tphi\circ\gamma=T_1\circ\tphi$ where $T_1:z\mapsto z+1$ is the translation.
\item a M\"obius transformation $A$
\end{itemize} 
satisfying
$$A\circ\DD\circ\tphi^{-1}=\iota:\Hyp\hookrightarrow\CPun.$$
A complex projective structure on $S_{g,n}$ is said to be \textit{parabolic} if it is parabolic at all the punctures. The set of parabolic type projective structures on $S_{g,n}$ is denoted by $\PP_{para}(S_{g,n})$.

\paragraph{Complex structures on the spaces $\PP_{para}(S_{g,n})$ and $\RR_{para}^{ne}(S_{g,n})$ and computation of their dimensions.\newline}\label{paradim}

Complex structure and dimension of $\PP_{para}(S_{g,n})$? For a more detailed explanation of the complex structure, we refer the reader to \cite[section 3]{Du}. Let us fix a complex structure $X$ on $S_{g,n}$, i.e. $X\in\TT(S_{g,n})$. The set of equivalent classes of complex projective structures on $S_{g,n}$ whose underlying complex structure is $X$ is parameterized by the set of holomorphic quadratic differentials on $X$. To see this, write $X=\Hyp/\Gamma$ where $\Gamma$ is a discrete subgroup of $\PSLR$. If $(\DD,\rho)$ is a complex projective structure on $S_{g,n}$ whose underlying complex structure is $X$, then $\DD:\Hyp\to\CPun$ is local biholomorphism. Hence, if $S_{\mathcal{D}}:=\left(\frac{\mathcal{D}''}{\mathcal{D}'}\right)'-\frac{1}{2}\left(\frac{\mathcal{D}''}{\mathcal{D}'}\right)^2$ is the Schwarzian derivative of $\DD$, by the $\rho$-equivariance of $\mathcal{D}$, the quadratic differential $S_{\mathcal{D}}(z)\cdot dz^2$ is invariant by the action of $\Gamma$ and then descends to a quadratic differential $\Phi$ on $X$. Moreover, it can be proved (see \cite[Theorem 9.1.1]{HPSG}) that the projective structure is parabolic at a puncture $p$ if and only if, in one (and hence in all) holomorphic coordinate sending $p$ to $0$, $\Phi$ has the following Laurent's series expansion:
$$\Phi=\left(\frac{1}{2z^2}+\frac{a_{-1}}{z}+a_0+\cdots\right)dz^2.$$
Hence if $X\in\TT(S_{g,n})$, the fiber $\pi^{-1}(X)$ of the forgetful map $\pi:\PP(S_{g,n})\to \TT(S_{g,n})$ is a complex affine space directed by the vectorial space of quadratic meromorphic differential on $\overline{X}$ (the compactification of $X$) with poles of order at most $1$ at the $n$ punctures and no other poles. By Riemann-Roch theorem, this set has complex dimension $3g-3+n$ (see \cite[paragraph 6.1]{DD}). As $\TT(S_{g,n})$ is also a complex manifold with complex dimension $3g-3+n$, we deduce that $\mathcal{P}(S_{g,n})$ is a complex manifold with complex dimension $6g-6+2n$. There is also a natural topology on $\PP(S_{g,n})$: we give the set of couples of maps $(\DD,\rho)$ the compact-open topology, and $\mathcal{P}(S_{g,n})$ inherits a quotient topology. As it is mentioned in \cite[section 3.3]{Du}, note that this topology is the topology coming from the complex structure previously defined.\newline 

Complex structure and dimension of $\RR_{para}^{ne}(S_{g,n})$? $\pi_1(S_{g,n})$ is a free group with $2g+n-1$ generators ($2$ generators $a_i$ and $b_i$ for each of the $g$ handles and $n-1$ generators $\alpha_i$ corresponding to $n-1$ punctures choosen among the $n$ punctures) and $\PSLC\sim SO_3(\mathbb{C})\subset\mathbb{C}^9$. Hence, the set $Hom(\pi_1(S_{g,n}),\PSLC)$ has the structure of an affine complex algebraic variety given by the identification:
\begin{align*}
Hom(\pi_1(S_{g,n}),\PSLC)&\longrightarrow\PSLC^{2g+n-1}\\
\rho&\longmapsto(\rho(a_1),\cdots,\rho(b_g),\rho(\alpha_1),\cdots,\rho(\alpha_{n-1}))
\end{align*}
This variety has complex dimension $3\cdot(2g+n-1)$. Define $Hom_{para}(\pi_1(S_{g,n}),\PSLC)$ as the set of representations $\rho$ such that $\rho(\alpha_i)$ is a parabolic matrix for $i\in\{1,\cdots,n\}$. This algebraic subset is given by $n$ algebraic equations ($Tr^2(\rho(\alpha_i))=4$), hence it has dimension $ 3\cdot(2g+n-1)-n=6g+2n-3$. Now $\PSLC$ acts algebraically on $Hom_{para}(\pi_1(S_{g,n}),\PSLC)$. Hence the quotient $\RR_{para}(S_{g,n})$ has $3$ dimensions less, i.e. it has dimension $3g+2n-6$. It is classical that the quotient $\RR_{para}(S_{g,n})$ is not Hausdorff but the subset $\RR^{ne}_{para}(S_{g,n})$ (consisting of these equivalence classes of representations which are non-elementary) is Hausdorff, hence is a complex manifold.\newline

Note that the holonomy of a parabolic complex projective structure is always non-elementary (see \cite[Lemma 10]{CDFG}). Hence we have the well defined map:
 $$Hol:\PP_{para}(S_{g,n})\longrightarrow\RR^{ne}_{para}(S_{g,n})$$
between complex manifolds of dimension $6g-6+2n$. According to \cite{Ea}, the map $Hol$ is holomorphic (see also \cite{Hu}).

\paragraph{Suspension.} Let $S$ be a smooth surface and $\rho:\pi_1(S)\to\PSLC$ be a group morphism. There is a classical construction which associates to $\rho$ a $\CPun$-fiber bundle over $S$ and a horizontal foliation. The construction is the following: $\pi_1(S)$ acts on the product $\widetilde{S}\times\CPun$ by $\gamma\cdot(\tilde{x},z)=(\gamma\cdot \tilde{x},\rho(\gamma)\cdot z)$.  This action is free and properly discontinuous. Hence the quotient $M_{\rho}=\widetilde{S}\times\CPun/\pi_1(S))$ is a manifold which fibers over $S$ via the projection $\Pi_{\rho}:[\tilde{x},z]\to[\tilde{x}]$ and whose fibers are copy of $\CPun$. The horizontal foliation on the product $\widetilde{S}\times\CPun$ whose leaves are the sets $\widetilde{S}\times\{z\}$ is globally preserved by the action of $\pi_1(S)$. Then, it descends to a foliation $\FF_{\rho}$ on $M_{\rho}$ transverse to the fibers. 

If there exists a smooth section $\sigma:S\rightarrow M_{\rho}$ of $\Pi_{\rho}$ transverse to the foliation, then $S$ is naturally endowed with a complex projective structure in the following way: the unique projective structure on any $\CPun$-fiber can be transported to $\sigma(S)$ by sliding along the leaves. This defines a complex projective structure on $\sigma(S)$ and thus on its $\Pi_{\rho}$ projection $S$.

Reciprocally, if $\rho$ is the holonomy representation associated to a complex projective structure $(\DD,\rho)$ on $S$, then the developing map $\DD$ gives a section $\sigma$ of the fiber bundle associated with $\rho$. This section is defined by:
$$\sigma([\tilde{x}])=[\tilde{x},\DD(\tilde{x})]$$
As $\DD$ is a local homeomorphism, this section is transverse to the horizontal foliation. Recall that a complex projective structure on $S$ naturally induces a complex structure on $S$. Taking the product with the complex structure of the $\CPun$-fibers, we also get a complex structure on $M_{\rho}$. And the section $\sigma:S\longrightarrow M_{\rho}$ is holomorphic (for the two complex structures which have just been defined on $S$ and $M_{\rho}$).

In the sequel, we will denote by $(M_{\rho},\Pi_{\rho},\FF_{\rho})$ the foliation associated with a representation $\rho:\pi_1(S)\to\PSLC$ and we will denote by $(M_{\rho},\Pi_{\rho},\FF_{\rho},\sigma)$ the pair \og foliation/section\fg associated with a complex projective structure $(\DD,\rho)$ on $S$.
 
\paragraph {Compactification.}\label{prel} Let $(M_{\rho},\Pi_{\rho},\FF_{\rho},\sigma)$ be the pair \og foliation/section \fg associated with a parabolic projective structure $(\DD,\rho)$ on $S_{g,n}$. As $S_{g,n}$ is not compact, the total space $M_{\rho}$ and the image of the section $\sigma$ are not compact. Following \cite{CDFG} or \cite{DD}, we are going to explain how we can compactifie the total space, the section and the foliation. The new foliation will have singularities. Recall that for any puncture, there is a neighborhood of this puncture $V=\bigslant{C}{<\gamma>}$, a homeomorphism $\tphi:C\to \Hyp$ satisfying $\tphi\circ\gamma=T_1\circ\tphi$ and a M\"obius transformation $A$ such that the following diagram is commutative:
$$
\begin{CD}
C @>\tphi>> \Hyp\\
@VV\DD V @VV\iota V\\
\CPun @>A>> \CPun
\end{CD}$$
Note that this implies $\rho(\gamma)=A^{-1}\circ T_1\circ A$. Indeed:
$$\rho(\gamma)\circ\DD=\DD\circ\gamma
=A^{-1}\circ\tphi\circ\gamma
=A^{-1}\circ T_1\circ\tphi
=A^{-1}\circ T_1\circ A\circ\DD.$$
Then, the following diagram is also commutative:
$$
\begin{CD}
\bigslant{C\times\CPun}{<\gamma>} @>[\tphi,A]>> \bigslant{\Hyp\times\CPun}{(\tau,z)\sim(\tau+1,z+1)}\\
@VV\Pi_{\rho} V @VV p_1 V\\
\bigslant{C}{<\gamma>} @>\phi=[\tphi]>> \bigslant{\Hyp}{\tau\sim\tau+1}
\end{CD}$$
Moreover, it can be easily checked that $\Phi:=[\tphi,A]$ sends $\FF_{\rho}$ on the horizontal foliation of $\bigslant{\Hyp\times\CPun}{(\tau,z)\sim(\tau+1,z+1)}$ and that $\Phi\circ\sigma=s_0\circ\phi$ where $s_0$ is the diagonal section: $s_0:=[\iota,\iota]: \bigslant{\Hyp}{\tau\sim\tau+1}\longrightarrow \bigslant{\Hyp\times\CPun}{(\tau,z)\sim(\tau+1,z+1)}$.\\

We are going to glue the horizontal foliation of the manifold 
$$\bigslant{\Hyp\times\CPun}{(\tau,z)\sim(\tau+1,z+1)}\sim \D^*\times\CPun$$
with a foliation of $\D\times\CPun$ using the following local model: consider the equation
\begin{equation}\label{eq-localmodel}
\frac{dv}{du}=-\frac{1}{2i\pi u},(u,v)\in \D\times\CPun
\end{equation}
The solutions of this equation define a foliation in $\D\times\CPun$ which is transverse to every fibre $\{u\}\times\CPun$ except the fibre $\{0\}\times\CPun$ which is an invariant projective line with only one singularitie at $(0,\infty)$. The leaf of the foliation through $(u_0,v_0)\in\D^*\times\CPun$ is the graph of the multivalued map:
$$v_{(u_0,v_0)}(u)=-\frac{1}{2i\pi}\log\left(\frac{u}{u_0}\right)+v_0$$
The map 
\begin{eqnarray*}
\bigslant{\Hyp\times\CPun}{(\tau,z)\sim(\tau+1,z+1)}&\longrightarrow&\D^*\times\CPun\\
(\tau,z)&\longmapsto&(e^{2i\pi\tau},z-\tau)
\end{eqnarray*}
is a biholomorphism which preserves the fibres and sends the horizontal foliation to the foliation defined by equation \eqref{eq-localmodel}. So we can glue these foliations together. Moreover, the diagonal section $s_0$ is sent to the constant section equal to $0$.

Doing this for any puncture, one gets a singular foliation $\overline{\FF}$ (with real dimension 2) on a compact manifold $\overline{M_{\rho}}$ (with real dimension 4) which is a $\CPun$-fiber bundle over the compact surface $S_g=\overline{S_{g,n}}$. The foliation is transverse to all the fibres except a finite number which are invariant projective complex lines with one singularitie. Moreover, the section $\sigma:S_{g,n}\rightarrow M_{\rho}$ associated with the developing map can be compactified as a section $\overline{\sigma}: S_g\rightarrow \overline{M_{\rho}}$ transverse to $\overline{\FF}$ and which does not pass through the singularities.

\begin{rem}\label{remarque}
Take any other smooth section $\overline{\sigma'}: S_g\rightarrow \overline{M_{\rho}}$ which is transverse to the foliation $\overline{\FF}$ and which does not pass through the singularities. The restriction $\sigma':=\overline{\sigma'}_{|S_{g,n}}$ defines a complex projective structure on $S_{g,n}$. This (new) complex projective structure is still parabolic. Indeed, for any puncture $p$, the vertical leaf of $\overline{\FF}$ above $p$ intersects transversely $\overline{S}=\overline{\sigma}(S_g)$ and $\overline{S'}=\overline{\sigma'}(S_g)$ in two points $a$ and $a'$. Then, flowing along the leaves naturally defines a diffeomorphism $h$ going from a neighborhood of $a$ in $\overline{S}$ to a neighborhood of $a'$ in $\overline{S'}$. 
The homeomorphism $\Psi:=\Pi_{\rho}\circ h\circ\sigma$ between two punctured neighborhood of $p$ in $S_{g,n}$ can be lifted, at the level of the universal cover, in a $<\gamma>$-equivariant homeomorphism $\widetilde{\Psi}$ which satisfies by construction $\DD'\circ\widetilde{\Psi}=\DD$ for some developing maps $\DD$ and $\DD'$ associated respectively to $\sigma$ and $\sigma'$. Hence, from the definition of a parabolic projective structure (see section \ref{Prel}), if the projective structure associated with $\sigma$ is parabolic, then the one associated with $\sigma'$ is also parabolic.
\end{rem}

\section{$Hol$ is locally injective.}
As the quotient $\RR_{para}^{ne}(S_{g,n})$ of non elementary representations modulo conjugacy is a Hausdorff complex manifold, it is enough to prove that if $(\DD_1,\rho)$ and $(\DD_2,\rho)$ are two parabolic complex projective structures on $S_{g,n}$ with the same holonomy $\rho$ and close enough, then they are equivalent, i.e. there exists $\tphi$ a $\pi_1(S_{g,n})$-equivariant homeomorphism of $\widetilde{S_{g,n}}$ such that $\DD_2=\DD_1\circ\tphi$. 

$(\DD_1,\rho)$ and $(\DD_2,\rho)$ being of parabolic type, for any puncture $p$ and for $i\in\{1,2\}$, one can find neighborhoods $V_i=\bigslant{C_i}{<\gamma>}$ of $p$ and homeomorphisms $\tphi_i:C_i\to\Hyp$ satisfying $\tphi_i\circ\gamma=T_1\circ\tphi_i$ such that the following diagram commutes:

$$
\begin{CD}
C_1 @>\tphi_1>> \Hyp @<\tphi_2<< C_2\\
@VV\DD_1 V @VV\iota V @VV\DD_2 V\\
\CPun @>A_1>> \CPun @<A_2<< \CPun
\end{CD}$$

\begin{lem}\label{lemme}
For all $t\in\R$, denote by $T_t$ the translation $z\mapsto z+t$. If we change the pair $\left(\tphi_2,A_2\right)$ in the previous diagram by $\left(\tphi_2^t,A_2^t\right):=\left(T_t\circ\tphi_2,T_t\circ A_2\right)$, we do not affect anything, i.e.:
\begin{enumerate}
 \item the diagram is still commutative: $A_2^t\circ\DD_2\circ(\tphi_2^t)^{-1}=\iota$
 \item $\tphi_2^t\circ\gamma=T_1\circ\tphi_2^t$.
\end{enumerate}
 Moreover, for one value of $t$, we have $A_2^t=A_1$.
\end{lem}

\begin{proof}
Items 1 and 2 are trivial. In order to prove the last assertion of the lemma, we denote $A^t:=A_1^{-1}\circ A_2^t$. Then, $A^t=(A_1^{-1}\circ A_2)\circ (A_2^{-1}\circ T_t\circ A_2)$. 
The matrix $A_1^{-1}\circ A_2$ commutes with $\rho(\gamma)$. Moreover $\rho(\gamma)$ is parabolic with $A_2^{-1}(\infty)$ as only fixed point. We deduce that $A_1^{-1}\circ A_2$ is either the identity or is parabolic with $A_2^{-1}(\infty)$ as only fixed point. Then for one value of $t$, we have: $A_2^{-1}\circ T_t\circ A_2=(A_1^{-1}\circ A_2)^{-1}$.
\end{proof}

Choosing the value of $t$ satisfying the previous lemma, we obtain the following commutative diagram:

$$\xymatrix{
C_1 \ar[r]^{\tphi_1} \ar[rd]_{\DD_1} &\Hyp \ar[d]^{\iota} &C_2 \ar[l]_{\tphi_2} \ar[ld]^{\DD_2}\\
&\CPun}$$

The map $\tphi:C_1\to C_2$ defined by $\tphi=\tphi_2^{-1}\circ\tphi_1$ is $<\gamma>$-equivariant and satisfies $\DD_1=\DD_2\circ\tphi$. We can continue $\tphi$ to all $proj^{-1}(V_1)=\bigsqcup_{\alpha\in\pi_1(S_{g,n})}\alpha\cdot C_1$ in the following way: for $x\in C_1$ and $\alpha\in\pi_1(S_{g,n})$, we define $\tphi(\alpha\cdot x)=\alpha\cdot\tphi(x)$. Then, we get a $\pi_1(S_{g,n})$-equivariant homeomorphism $\tphi:proj^{-1}(V_1)\to proj^{-1}(V_2)$ which still satisfies the equality $\DD_1=\DD_2\circ\tphi$\footnote{Lemma \ref{lemme} was important to keep the equality satisfied}.

The previous construction can be done for any puncture $p$. We denote, for $i\in\{1,2\}$, by $V_i(p)=C_i(p)/<\gamma>$ the neighborhoods of the puncture $p$ and $\tphi_p:proj^{-1}(V_1(p))\to proj^{-1}(V_2(p))$ the homeomorphisms as in the previous part. Let $V_1'(p)=C_1'(p)/<\gamma>$ be a smaller neighborhood of $p$ (i.e. $\overline{V_1'(p)}\subset V_1(p)$) and take $C_2'(p)=\tphi_p(C_1'(p))$ and $V_2'(p)=proj(C_2'(p))$. Denote also by $K_i=S_{g,n}\setminus\bigsqcup_p V_i'(p)$ and by $\widetilde{K_i}=proj^{-1}(K_i)$. We are going to prove that there is a $\pi_1(S_{g,n})$-equivariant homeomorphism $\tpsi:\widetilde{K_1}\to \widetilde{K_2}$ satisfying $\DD_1=\DD_2\circ\tpsi$.

For this, let $(M_{\rho},\Pi_{\rho},\FF_{\rho})$ be the foliation associated with the representation $\rho$. The developing maps $\DD_1$ and $\DD_2$ induce two sections $\sigma_1$ and $\sigma_2$ of the bundle.
As the $K_i$ are compact, if $(\DD_1,\rho)$ and $(\DD_2,\rho)$ are close enough, then there exists a holonomy map $h:\sigma_1(K_1)\to \sigma_2(K_2)$ close to the identity. Now, the map $\psi:K_1\to K_2$ defined by $\psi=\Pi_{\rho}\circ h\circ\sigma_1$ can be lifted uniquely in a $\pi_1(S_{g,n})$-equivariant homeomorphism $\tpsi:\widetilde{K_1}\to \widetilde{K_2}$ satisfying $\DD_1=\DD_2\circ\tpsi$.

To conclude, we just have to prove that for all the punctures $p$, the maps $\tphi_p$ and $\tpsi$ coincide on their intersection $proj^{-1}(V_1(p)\setminus V_1'(p))$. Let $x\in C_1(p)\setminus C_1'(p)$, we have:
$$\DD_1(x)=\DD_2(\tphi_p(x))=\DD_2(\tpsi(x)).$$
As $\tphi_p(x)$ and $\tpsi(x)$ belong to $C_2(p)\setminus C_2'(p)$ and $\DD_2$ is injective on $C_2(p)$, we deduce that $\tphi_p(x)=\tpsi(x)$. As the maps $\tphi_p$ and $\tpsi$ are both $\pi_1(S_{g,n})$-equivariant, we deduce that they coincide on $proj^{-1}(V_1(p)\setminus V_1'(p))$.\\

We have seen in section \ref{paradim} that $\PP_{para}(S_{g,n})$ and $\RR_{para}^{ne}(S_{g,n})$ are two complex manifolds with the same dimension and that $Hol$ is a holomorphic map. We deduce that $Hol$ is open. We have just proved that $Hol$ is locally injective. Hence $Hol$ is a local biholomorphism and theorem \ref{theoprincipal} is proved. Nevertheless, in this last section, we are going to give a geometric proof of the fact that $Hol$ is open without using the fact that it is holomorphic.

\section{$Hol$ is open.}
Let $(\DD_0,\rho_0)$ be a parabolic complex projective structure on $S_{g,n}$. We have to prove that, if $\rho$ is close enough to $\rho_0$, then $\rho$ is the holonomy representation of a parabolic complex projective structure.

In the case where the surface is compact, the proof is classical (see for example \cite[Theorem 1.3]{LM} or \cite[Proposition 1.7.1]{Me}), the idea being the following: for every perturbation $\rho$ of $\rho_0$, construct the suspended foliation $(M_{\rho},\Pi_{\rho},\FF_{\rho})$ associated to $\rho$. If $\rho$ is close to $\rho_0$, then $(M_{\rho},\Pi_{\rho},\FF_{\rho})$ is \og close \fg to $(M_{\rho_0},\Pi_{\rho_0},\FF_{\rho_0})$ and the section $\sigma_0$ of $(M_{\rho_0},\Pi_{\rho_0},\FF_{\rho_0})$ associated to the developing map $\DD_0$ (transverse to $\FF_0$) can be transported to define a section $\sigma$ of $(M_{\rho},\Pi_{\rho},\FF_{\rho})$ which is still transverse to $\FF_{\rho}$. This section defines a developing map $\DD$ with holonomy $\rho$. 

Trying to adapt this proof to our case, we are faced with two difficulties. Firstly, the base surface is not compact anymore. Hence the perturbated section $\sigma$ could be not transverse to the foliation anymore even if the perturbation is arbitrarily small. Secondly, the projective structure associated with the perturbated section could be not parabolic anymore. The idea of the proof is to compactifie the bundles as explained in section \ref{prel}.


More precisely, let $U_{\rho_0}$ be a contractible neighborhood of $\rho_0$ in the set of all parabolic representations of $\pi_1(S_{g,n})$. The group $\pi_1(S_{g,n})$ acts on $\widetilde{S_{g,n}}\times\CPun\times U_{\rho_0}$ by:
$$\alpha\cdot(\widetilde{x},z,\rho)\longmapsto(\alpha\cdot\widetilde{x},\rho(\alpha)\cdot z,\rho).$$
The quotient is a $\CPun$-bundle $E$ over $S_{g,n}\times U_{\rho_0}$. It is endowed with a foliation: the one whose leaves are the quotients of the sets $\widetilde{S_{g,n}}\times\{z\}\times \{\rho\}$. For every $\rho\in U_{\rho_0}$, the restriction $E_{\rho}$ of $E$ over $S_{g,n}\times\{\rho\}$ is the suspension of $\rho$.

As $(\DD_0,\rho_0)$ is parabolic, for any puncture, there is a neighborhood $\bigslant{C}{<\gamma>}$, a homeomorphism $\tphi:C\to \Hyp$ satisfying $\tphi\circ\gamma=T_1\circ\tphi$ and a M\"obius transformation $A_0$ such that:
$$A_0\circ\DD_0\circ\tphi^{-1}=\iota:\Hyp\hookrightarrow\CPun.$$

For every $\gamma\in\pi_1(S_{g,n})$ which makes one loop around a puncture and for every $\rho\in U_{\rho_0}$, the M\"obius transformation $\rho(\gamma)$ is parabolic. So, we can find an analytic family of M\"obius transformations $\{A_{\rho};\rho\in U_{\rho_0}\}$ satisfying $A_{\rho_0}=A_0$ and $A_{\rho}\circ\rho(\gamma)\circ A_{\rho}^{-1}=T_1$. Then we have the following commutative diagram:

$$
\begin{CD}
\bigslant{C\times\CPun\times U_{\rho_0}}{<\gamma>} @>[\tphi,A_{\rho},id]>> \bigslant{\Hyp\times\CPun}{(\tau,z)\sim(\tau+1,z+1)}\times U_{\rho_0}\\
@VVV @VVV\\
\bigslant{C}{<\gamma>}\times U_{\rho_0} @>([\tphi],id)>> \bigslant{\Hyp}{\tau\sim\tau+1}\times U_{\rho_0}
\end{CD}$$

We can compactifie using the local model explained in the preliminaries. We get a $\CPun$-fiber bundle $\overline{E}$ over $S_g\times U_{\rho_0}$. By hypothesis, the compactified section $\overline{\sigma_0}$ over $S_g\times\{\rho_0\}$ is transverse to the foliation and does not meet any singularities. The open set $U_{\rho_0}$ being contractible, by the universal property of the fibre bundles, $\overline{\sigma_0}$ can be extended to a global section of $\overline{E}$ over $S_g\times U_{\rho_0}$. For $\rho$ close enough to $\rho_0$, the restriction $\overline{\sigma_{\rho}}$ of this global section to $S_g\times\{\rho\}$ is still transverse to the foliation and does not meet any singularities. Then, as explained in the preliminaries (see in particular remark \ref{remarque}), the restriction $\sigma_{\rho}$ of $\overline{\sigma_{\rho}}$ to $S_{g,n}\times\{\rho\}$ defines a complex projective structure of parabolic type on $S_{g,n}$ with holonomy $\rho$.


\end{document}